\documentclass[11pt,reqno]{amsart} 
\usepackage{amsmath,amssymb,rawfonts}
\usepackage{tikz}
\usepackage{mathrsfs}
\usepackage{tikz-cd}
\usetikzlibrary{cd}
\usepackage{bm}
\usetikzlibrary{arrows, matrix}
\usepackage{mathtools}
\usepackage{stmaryrd}

\newtheorem{theorem}{Theorem}[section]

\newtheorem{defi}[theorem]{Definition}
\newtheorem{prop}[theorem]{Proposition}
\newtheorem{coro}[theorem]{Corollary}

\newtheorem{remark}[theorem]{Remark}
\newtheorem{ex}[theorem]{Example}

\title{GROTHENDIECK-TOPOLOGICAL GROUP OBJECTS}
\author{Joaqu\'in Luna-Torres}
\address{Programa de Matem\'aticas, Universidad Distrital Francisco Jos\'e de Caldas,  Bogot\'a D. C., Colombia (retired professor)}
\dedicatory{In memory of Carlos J. Ruiz Salguero}
\email{ jlunator@fundacionhaiko.org}

\begin{document} 
\maketitle 
\begin{abstract}
In analogy with the classical theory of topological groups, for  finitely complete categories enriched with Grothendieck topologies, we provide  the concepts of localized $\mathfrak{G}$-topological space, initial Grothendieck topologies and continuous morphisms,  in order to obtain the concepts of $\mathfrak{G}$-topolo\-gical monoid and $\mathfrak{G}$-topological group objects.

\end{abstract}

\section{ Introduction}

Topological groups (monoids) are very important mathematical objects with not only applications in mathematics, for example in Lie group theory, but also in physics. Topological groups are defined on topological spaces that involve points as fundamental objects of such topological spaces. 

The aim of this paper is to introduce a topological theory of groups  within any category, finitely complete, enriched with Grothendieck topologies.

The paper is organized as follows: We describe, in section $2$, the notion of Grothendieck topology as in S. MacLane and I. Moerdijk \cite{MM} . In section $3$, we present the concepts of localized $\mathfrak{G}$-topological space, initial Grothendieck topologies and continuous morphisms, and we study the lattice structure of all Grothendieck topologies on a category $\mathscr{C} $; after that, in section $4$, as in M. Forrester-Baker \cite{MFB}, we present the concepts of monoid and group objects. In section $5$ we present the concepts of {$\mathfrak{G}$-topological monoid and {$\mathfrak{G}$-topological group objects; finally, in section $6$ we give some examples of these ideas.
\section{Theoretical Considerations}
Throughout this paper, we will work within an ambient category $\mathscr{C} $ which is finitely complete.
From S. MacLane and I. Moerdijk \cite{MM}, Chapter III, we have the following:
\begin{itemize}
\item Let $\mathscr{C} $ be a small category, and let $Sets^{\mathscr{C}^{op}}$  be the corresponding functor category and let 
\begin{align*}
y: &\mathscr{C} \rightarrow  Sets^{\mathscr{C}^{op}}\\
&C\mapsto \mathscr{C}(-,C)
\end{align*}
then a sieve $\mathcal S$ on $\mathscr{C}$ is simply a subobject  $\mathcal S\subseteq y(C)$ in $Sets^{\mathscr{C}^{op}}$.

Alternatively, $\mathcal S$ may be given as a family of morphisms in $\mathscr{C}$, all with codomain $C$, such that
\[
f \in  \mathcal S \Longrightarrow f\circ g \in \mathcal S
\]
whenever this composition makes sense; in other words, $\mathcal S$ is a right ideal. 

\item If $\mathcal S$ is a sieve on $\mathscr C$ and $h: D\rightarrow C$ is any arrow to $C$, then  $h^{*}(\mathcal S) = \{g \mid cod(g) = D, h\circ g \in \mathcal S\}$  is a sieve on $D$.

\item A Grothendieck topology on a category $\mathscr{C} $ is a function $J$ which assigns to each object $C$ of $\mathscr{C} $ a collection $J(C) $ of sieves on $\mathscr{C} $, called {\bf{local G-topology on $C$} }, in such a way that
\begin{enumerate}
\item [(i)] the maximal sieve $ t_c = \{f \mid cod(f) = C\}$  is in $J(C)$;
\item[(ii)] (stability axiom) if $\mathcal S \in J(C)$, then $h^{*}(\mathcal S) \in J(D)$ for any arrow $h:D\rightarrow C$;
\item[(iii)] (transitivity axiom) if $\mathcal S \in J(C)$ and $\mathcal R$ is any sieve on $C$ such that $h^{*}(\mathcal R) \in J(D)$ for all $h: D\rightarrow C$ in $\mathcal S$, then $\mathcal R \in J(C)$.
\end{enumerate}
\item A site will mean a pair $(\mathscr C, J)$ consisting of a small category $\mathscr C$ and a Grothendieck topology $J$  on $\mathscr  C$. 
\item If $\mathcal S \in J(C)$, We say that $\mathcal S$ is a covering sieve, or that $\mathcal S$ covers $C$. We will also say that a sieve $\mathcal S$ on $\mathscr C$ covers an arrow $f: D \rightarrow  C$ if $f^{*} (\mathcal S)$ covers $D$. (So $\mathcal S$ covers $C$ iff $\mathcal S$ covers the identity arrow on $C$).
\end{itemize}

\section{Initial Grothendieck topologies and continuous morphisms}
In general topology and related areas of mathematics, the initial topology (or weak topology or limit topology or projective topology) on a set  $X$, with respect to a family of functions on $X$, is the coarsest topology on $X$ that makes those functions continuous (see N. Bourbaki \cite{NB}). These topologies are used, for example,  in the study of subspace topology,  product topology, The inverse limit of any inverse system of spaces and continuous maps and The weak topology on a locally convex space.

In this section, we are interested in studying initial topologies in some categories enriched with Grothendieck topologies.

\begin{defi}
Given a category $\mathscr C$, a Grothendieck topology $J$ on  $\mathscr C$and $X$ an object of $\mathscr C$, we say that $J(X)$ is a  {\bf localized {$\mathfrak{G}$-topology on}} $X$ and that $(X,J(X))$  is a {\bf localized $\mathfrak{G}$-topological space}.
\end{defi}
\begin{prop}
Let $\mathscr C$ be a finitely complete category,  and $J$ be a Grothendieck topology on  $\mathscr C$. For a morphism $f: B\rightarrow C$ in $\mathscr C$ and a sieve $\mathcal S \in J(C)$, let $f^{*}(\mathcal S) = \{\psi \mid cod(\psi) = B, f\circ \psi \in \mathcal S\}$ be the corresponding sieve on $B$, then  $f^{*}(J(C))=: \{g^{*}(\mathcal S)\mid \mathcal S \in J(C)\}$ is a   localized $\mathfrak{G}$-topology on $B$.
\end{prop}

\begin{proof}
\begin{enumerate}
\item [(i)] Since $f$ is in the maximal  sieve $ t_c$ then for any arrow $\alpha$ with codomain $B$, \ $\phi\circ \alpha \in t_C$, therefore $\phi^{*}(t_C)=t_B$.
\item[(ii)] Suppose $S\in f^{*}(J(C))$ 
then $T=\{f\circ h\mid h\in S\}$ is a sieve in $J(C)$, therefore $(g\circ f)^{*}(T)\in J(A)$. Since, in the following diagram, the right square and the outer rectangle are pullbacks, so is the left square:

\begin{equation}\label{pb}
\begin{tikzcd}
 (f\circ g)^{*}(S)\arrow{r}  \arrow[tail]{d} & f^{*}(S) \arrow[tail]{d} \arrow{r} &  S \arrow[tail]{d}  \\
\mathscr C(-,A) \arrow{r} & \mathscr C(-,B)\arrow{r}& \mathscr C(-,C)\\
\end{tikzcd}
\end{equation}
(where the bottom edges are the natural transformations inducided by $g$ and $f$ respectively). (See S. MacLane \cite{SM}, pg 72).

For this reason,
\[
(g\circ f)^{*}(T)=g^{*}\left(f^{*}(T)\right)=g^{*}(S)\,\ \text{is a sieve in} \,\ J(A).
\]
\item[(iii)]  Let $R$ be any sieve on $B$ and $S\in f^{*}(J(C))$. Then $T=\{f\circ h\mid h\in S\}$ is a sieve in $J(C)$ and  $R'=\{f\circ r\mid r\in R\}$ is a sieve in $C$, using again diagram (\ref{pb}) we have $R' \in J(C)$, in other words, $R\in f^{*}(J(C))$.
\end{enumerate}
\end{proof}

\begin{defi}
For $J$ a Grothendieck topology on a category $\mathscr{C} $,
\begin{itemize}
\item a morphism $f:B\rightarrow C$ of $\mathscr C$ is called {\bf $\mathfrak{G}$-continuous} if  the  {\bf localized {$\mathfrak{G}$-topology on}} $B$,\,\  $f^{*}(J(C))$ satisfies
\[
J(B)\subseteq  f^{*}(J(C)).
\]
\item 
 $f^{*}(J(C))$ is the coarsest {$\mathfrak{G}$}-topology on $B$ for which $f$  is  {\bf $\mathfrak{G}$}-continuous and it is called  {\bf initial{$\mathfrak{G}$-topology on}} $B$.
\end{itemize}
\end{defi}

\begin{prop} \label{comp}
 Let $f:X\rightarrow Y$ and $g:Y\rightarrow Z$ be two morphisms  of  $\mathscr C$ $\mathfrak{G}$-continuous then $g\centerdot f$ is a  morphism of  $\mathscr C$  which is $\mathfrak{G}$-continuous.
 \end{prop}
 \begin{proof}
 Since $f:X\rightarrow Y$ and $g:Y\rightarrow Z$ are $\mathfrak{G}$-continuous, we have 
\[
J(X)\subseteq f^{*}(J(Y)) \,\  \text{and}\,\  J(Y)\subseteq g^{*}(J(Z)),
\]

therefore 
\[
f^{*}(J(Y))\subseteq f^{*}(g^{*}(J(Z)))
\]
then
\[
J(X)\subseteq f^{*}(g^{*}(J(Z)))= (g\circ f)^{*}(J(Z)).
\]
This complete the proof.   
 \end{proof}
 As a consequence we obtain
 \begin{defi}
 The category {\bm{${\mathfrak{G}-Top} _{\small{\mathscr C}}$}} of localized $\mathfrak{G}$-spaces comprises the following data:
 \begin{enumerate}
 \item {\bf Objects}: Localized $\mathfrak{G}$-topology spaces $(X,J(X))$
\item {\bf Morphisms}: Morphisms of $\mathscr C$ which are $\mathfrak{G}$-continuous.
 \end{enumerate}
 \end{defi}

\subsection{The lattice structure of all Grothendieck topologies}
 For a category ${\mathscr C}$  we consider the set $\{J_{\alpha}\mid \alpha \in A\}$. Using Lemma $0.34$ of P. T. Johnstone \cite{PJ2}, we obtain:

 \begin{coro}\label{complete}
  For every  object $X$ of $\mathscr C$
  \[
  \{J_{\alpha}(X)\mid \alpha \in A\}
  \]
  is a complete lattice.
 \end{coro}
Using the definition of  topological functor in G.C.L. Br\"ummer \cite{GB}, we have:

\begin{theorem}
The forgetful functor $U: {\bm{{\mathfrak{G}-Top} _{\small{\mathscr C}}}} \rightarrow  \mathscr C$ is a topological functor.
\end{theorem}

\begin{proof}
Let $X$ be an object of $\mathscr C$ , let  $(Y_i,J(Y_i))_{i\in I}$ be a family of objects of {\bm{${\mathfrak{G}-Top} _{\small{\mathscr C}}$}}, and  for each $i\in I$ let $f_i: X\rightarrow Y_i$  be a morphism of $\mathscr C$. We must show that for each object  $(Z,J(Z))$ 
of {\bm{${\mathfrak{G}-Top} _{\small{\mathscr C}}$}} and for each morphism $g:Z\rightarrow X$  of $\mathscr C$, $g$ is $\mathfrak{G}$-continuous if and only if each of the morphisms $f_i\circ g$ of $\mathscr C$ is $\mathfrak{G}$-continuous, when the object $X$ comes together with the localized $\mathfrak{G}$-topology $\hat J(X)=\displaystyle{\bigcap_{i\in I}}f_i^{*}(J(Y_i))$.
The morphisms $f_i;(X,\hat J(X))\rightarrow ((Y_i,J(Y_i))$ are evidently continuous, since $\hat J(X)\subseteq f_i^{*}(J(Y_i))$, for al $i\in I$, and from proposition (\ref{comp}) it readily follows that each morphism  $f_i\circ g$ of $\mathscr C$ is $\mathfrak{G}$-continuous whenever  $g$ is.

For the converse, suppose that each of the morphisms $f_i\circ g$ of $\mathscr C$ is $\mathfrak{G}$-continuous, since
\[
\begin{split}
J(Z) &\subseteq \bigcap_{i\in I}(f_i\circ g)^{*}(J(Y_i))\\
&=\bigcap_{i\in I} g_{*}((f_i)^{*}(J(Y_i))\\
&=g_{*}(\bigcap_{i\in I} (f_i)^{*}(J(Y_i))\\
&=g_{*}(\hat J(X)),
\end{split}
\]
we conclude that the morphism $g$  is $\mathfrak{G}$-continuous.
\end{proof}

\section{Monoid  and group objects}
Following M. Forrester-Baker \cite{MFB}, in a category $\mathscr{C} $ with  finite products (and hence with a  terminal object  $1$),  we can thus define:
\begin{enumerate}
\item A monoid  object $(G,\mu,\eta)$ of $\mathscr{C} $  to be an object $G$ together with the following  two arrows 
\[
\mu: G\times G\longrightarrow G,\,\,\,\,\,\,\,\ \eta: 1\longrightarrow G
\]
such that the following  two diagrams commute:
\vspace{1cm}
\begin{center}
$
\begin{tikzcd}
G\times G\times G \arrow{r} {1_{\text{\tiny G}}\times\mu} \arrow{d}[swap] {\mu \times 1_{\text{\tiny G}}}
& G\times G \arrow{d} {\mu} \\
G\times G \arrow{r}[swap]{\mu} & G\\
\end{tikzcd}
$
$
\begin{tikzcd}
1\times G\arrow{rd}[swap]{\pi_1} \arrow{r}{\eta\times 1_{\text{\tiny G}}} & G\times G \arrow{d}{\eta} & G\times 1  \arrow{l}[swap]{1_{\text{\tiny G}}\times \eta} \arrow{ld}{\pi_2}\\
  & G
\end{tikzcd}
$
\end{center}
here $1_{\text{\tiny G}}$ in $1_{\text{\tiny{G}}}\times \mu$ is the identity arrow $G\rightarrow G$, while $\pi_1, \pi_2$ are the respective projections.
\item A group object $(G,\mu,\eta, \zeta)$ is a monoid object equipped with an arrow $\zeta: G\rightarrow G$ such that the following diagram commutes:

\begin{center}
$
\begin{tikzcd}
 G \arrow[r, "\delta" ] \arrow[d] & G\times G \arrow[r, "1_{\text{\tiny G}}\times\zeta"] & G\times G \arrow [d, "\mu"]\\
1  \arrow[rr]&& G
\end{tikzcd},
$
\end{center}

where $\delta$ is the unique arrow making the following diagram commutes:

\[
 \begin{tikzcd}
   & G\arrow{ld}[swap]{1_{\text{\tiny G}}} \arrow{rd}{1_{\text{\tiny G}}} \arrow{d}{\delta}&  \\
    G &G\times G \arrow{l}{\pi_1} \arrow{r}[swap]{\pi_2}&G. \\
 \end{tikzcd}
\]

\item  An abelian  group object is a group object $(G,\mu,\eta, \zeta)$  completed by  the twist morphism $\tau: G\times G\rightarrow G\times G$ for which $\pi_1=\pi_2\circ \tau$  and $\pi_2=\pi_1\circ\tau$.

\end{enumerate}
\begin{defi}

We say that a homomorphism $f:(G_1,\mu_1,\eta_1, \zeta_1)\rightarrow (G_2,\mu_2,\eta_2, \zeta_2) $ of group objects is a morphism $f: G_1\rightarrow G_2$ of  $\mathscr{C} $ such that the following diagram commutes:

\[
\begin{tikzcd}
G_1\times G_1 \arrow{r} {f\times\ f} \arrow{d}[swap] {\mu_1 }
& G_2\times G_2 \arrow{d} {\mu_2} \\
G_1\arrow{r}[swap]{f} & G_2\\
\end{tikzcd}
\] 
\end{defi}
\begin{prop}\label{GO}
Let $f:(G_1,\mu_1,\eta_1, \zeta_1)\rightarrow (G_2,\mu_2,\eta_2, \zeta_2) $ and $g:(G_2,\mu_2,\eta_2, \zeta_2)\rightarrow (G_3,\mu_3,\eta_3, \zeta_3) $ be homomorphisms of group objects, then $g\circ f:(G_1,\mu_1,\eta_1, \zeta_1)\rightarrow (G_3,\mu_3,\eta_3, \zeta_3) $ is a homomorphism of group objects.
\end{prop}
\begin{proof}
By hipothesis, left and right squares of the following diagram are commutative
\begin{equation}\label{pb}
\begin{tikzcd}
 G_1\times G_1\arrow{r}{f\times f}  \arrow{d}{\mu_1} & G_2\times G_2 \arrow{d}{\mu_2} \arrow{r}{g\times g} &  G_3\times G_3 \arrow{d}{\mu_3}  \\
 G_1 \arrow{r}{f} & G_2\arrow{r}{g}& G_3 \\
\end{tikzcd}
\end{equation}
then so is the outer square.
\end{proof}
As an immediate consequence the group objects form a subcategory of $\mathscr{C} $.

\section {$\mathfrak{G}$-topological group objects}
\begin{defi}
Let $\mathscr C$ be a category with  finite products,  and $J$ be a Grothen\-dieck topology on  $\mathscr C$. A group object $G$ in $\mathscr C$ is called a  {\bf G-topological group object of $(\mathscr C, J)$} if the morphisms
\[
\mu: G\times G\longrightarrow G,\,\,\,\,\,\,\,\ \zeta: G\rightarrow G
\]
are {\bf $\mathfrak{G}$-continuous} in the  localized $\mathfrak{G}$-topology space  $(G,J(G))$.
\end{defi}

\begin{remark}
To prove that $\mu$ is $\mathfrak{G}$-continuous, it is enough to require that $G\times G$ carry the localized {$\mathfrak{G}$-topology $\pi_1^{*}(J(G)) \cap \pi_2^{*}(J(G))$.}
\end{remark}

The proof of the following theorem  is a direct  consequence of propositions \ref{comp} and \ref{GO}

\begin{theorem}
$\mathfrak{G}$-topological group objects of $(\mathscr C, J)$ constitute a subcategory $\pmb{{\mathfrak{G}-Top-Gr} _{\small{\mathscr C}}}$  of the category of group objects of $\mathscr C$.
\end{theorem}

\section{Examples of monoid and group objects}
\begin{ex}
As in  \cite{JLT}, the divisibility  partial ordering on $\mathbb Z^{+}$ is defined as follows: for all $k,n\in \mathbb Z^{+}$,\ $k\leqslant n$ if and only if there exists $t\in \mathbb Z^{+}$ such that $n=kt$.

We will denote by $\mathbb Z^{+}_{D}$ the partially ordered set $(\mathbb Z^{+}, \leqslant)$ and we can observe

\begin{itemize}
\item The down-set generated by $n\in \mathbb Z^{+}$ is $\downarrow n = \{k \in \mathbb Z^{+} : k \leqslant n\}= \mathfrak D_n$, the set of all divisors of $n$.
\item If $M\subseteq  \mathbb Z^{+}$ then the down-set generated by $M$ is $\downarrow M =\bigcup_{m\in M} \mathfrak D_m$.

\item Using  S. MacLane and  I. Moerdijk  results (see \cite{MM}), we have that a lattice $L$ is a partially ordered set which, considered as a category, has all binary products and all binary coproducts. In this way, if we write $m, n$ for objects of $L= \mathbb Z^{+}_{D}$ then $m\leqslant n$ if and only if there is a unique arrow $m\rightarrow n$, the coproduct of $m$ and $n$ (i.e. their pushout) is the least upper bound, or least common multiple, $m\vee n= l.c.m\{m,n\}$ and the product (i.e. their pullback) is the greatest lower bound, or greatest common divisor, $m\land n=g.c.d\{m,n\}$. 

\item It is clear that $\mathbb Z^{+}_{D}$ is a distributive lattice.
\item $(\mathbb Z^{+}_{D}, LCM, 1)$, where the morphism $ lcm: \mathbb Z^{+}_{D}\times \mathbb Z^{+}_{D}\rightarrow  \mathbb Z^{+}_{D}$ is defined by $(a,b)\mapsto l.c.m\{a,b\}$, is a {\bf monoid object}.
\item For each object $n$ of  $\mathbb Z^{+}_{D}$ we have a monoid $(\mathfrak D_n, LCM,1)$. Actually, this is a submonoid of  $(\mathbb Z^{+}_{D}, lcm, 1)$ since
\[
\begin{tikzcd}
\mathfrak D_n\times \mathfrak D_n\arrow{r} {i\times\ i} \arrow{d}[swap] {lcm }
& \mathbb Z^{+}_{D}\times \mathbb Z^{+}_{D} \arrow{d} {lcm} \\
\mathfrak D_n\arrow{r}[swap]{i} & \mathbb Z^{+}_{D}\\
\end{tikzcd}
\]
 is a commutative square, where $i$ is the inclusion map.
\end{itemize}
A Grothendieck topology  on  the category $\mathbb Z^{+}_D$ is a function $J$ which assigns to each object $n$ of  $\mathbb Z^{+}$ a collection $J(n)$ of sieves on $\mathbb Z^{+}_D$, in such a way that
\begin{enumerate}
\item[(i)] the maximal sieve $\mathfrak D_n$ (the set of all divisors of $n$) is in $J(n)$;
\item[(ii)] (stability axiom) if $S\in J(n)$ and $k\leqslant n$ then $S\cap\mathfrak D_k$ is in $J(k)$;
\item[(iii)] (transitivity axiom) if $S\in J(n)$ and $R$ is any sieve on $n$ such that $R\cap \mathfrak D_k$ is in $J(k)$ for each $k\in S$, then $R\in J(n)$.
\end{enumerate}

Some Grothendieck topologies on  $\mathbb Z^{+}_D$ are the following
\begin{itemize}

\item The discrete Grothendieck topology on $\mathbb Z^{+}_D$ is given by\linebreak $J_{dis}(n) = \mathcal D(\mathfrak D_n) $.
\item A subset $D\subseteq \mathfrak D_n$ is said to be dense below $n$ if for any $m\leqslant n$ there is $k\leqslant m$ with $k\in D$. The dense topology on $\mathbb Z^{+}_D$ is given by $J(n)= \{ D\mid k\leqslant n,\,\ \forall\,\ k\in D,\,\ \text{$D$ is a sieve dense below}\,\ n\}$ 
\end{itemize}
For these topologies, the previous monoids are {\bf $\mathfrak{G}$-topological monoids}.
\end{ex}
\begin{ex}
From G. C. Wraith \cite{GCW}, we have that the  notion of locale is straightforward to formulate in a topos, and it is now well established that the localic approach is a right way to do topology in a topos \cite{PJ1}.

To set out notation, let us denote by $Loc(\mathscr{E})$ the category of locales in the topos $\mathscr{E}$. 

If $A$ and $B$ are objects of a topos $\mathscr{E} $  we can form the locale $Map ( A, B )$ in $ \mathscr{E} $ by taking an $A \times B$-indexed family of generators
\[
<a \mapsto b>, \hspace{1cm} a\in A,\,\ b\in B
\]
satisfying the relations

\[
<a \mapsto b> \land   <a \mapsto b'>\,\ \leqslant \,\ h^{*}(b= b'), \,\,\,\,\,\,\ \bigvee_{b}<a \mapsto b>=\top.
\]

By adding the further relations
\[
<a \mapsto b> \land   <a' \mapsto b>\,\ \leqslant\,\ h^{*}(a= a')
\]
we get a locale $Inj (A, B )$ , or by adding the further relations
\[
\bigvee_{a}<a \mapsto b>=\top,
\]
we get a locale $Surj(A, B )$ , or by adding both we get a locale $Bij(A, B )$ .

We define $Perm ( A )$ to be $Bij(A, A )$ ; it is clearly a localic group, i. e. {\bf a group object in $ Loc(\mathscr{E})$}. Let us describe its structure in slightly more detail; It has an identity
\[
e: \Omega_{\mathscr{E}}\rightarrow Perm(A)\,\,\,\  \text{given by}\,\,\,\ e^{*}<a \mapsto a'>=  h^{*}(a= a').
\]

It has a multiplication
\[
m: Perm(A)\times Perm(A)\rightarrow Perm(A)
\]
given by
\[
m^{*}<a \mapsto a''>=\bigvee_{a'}\big(p_1^{*}<a \mapsto a'>\,\ \land\,\ p_2^{*}<a' \mapsto a''>\big)
\]
where $p_l, p_2$ are the projections from the product $Perm(A) \times Perm (A)$.

The inverse map is clearly given by
\[ 
<a \mapsto a'>\,\ \mapsto\,\ <a' \mapsto a> .
\]
It is straightforward to show that $(Perm ( A ),m ,e)$ is a {\bf $\mathfrak{G}$-topological group object}  with each of the following Grothendieck topologies:
\begin{itemize}
\item (The atomic topology.) Let $\mathscr C$ be a category, and define $J$  by $S \in J (C)$ iff the sieve $S$ is nonempty. 
\item (The dense topology.) Let $\mathscr C$ be a category, for a  sieve $S$, let

``$S\in J(C)$ iff for any $f:D\rightarrow C$ there is a $g:E\rightarrow D$ such that $f\circ g \in S$". 
\item The discrete Grothendieck topology on  $Loc(\mathscr{E})$.
\end{itemize}
\end{ex}
\section*{Conclusions and Comments}
We proposed the construction of initial Grothendieck topologies and continuous morphisms and a new concept of {$\mathfrak{G}$-topological group objects}. The properties  of these objects (similar to those of the topological groups) are  postponed to future works.
\section*{Acknowledgments}.
 The author wish to thank \emph{ Fundaci\'on Haiko} of Colombia for its constant encouragement during the development of this work.

\end{document}